\documentclass[10pt]{amsart}
\usepackage{amssymb}
\usepackage{amsmath}
\usepackage{amsthm}
\usepackage{amscd}
\usepackage{amstext}
\usepackage{amsfonts}
\usepackage[all]{xy}

\theoremstyle{plain}
\newtheorem{theo}{Theorem}[section]
\newtheorem{prop}[theo]{Proposition}
\newtheorem{lem}[theo]{Lemma}

\theoremstyle{remark}
\newtheorem{defi}[theo]{\rm Definition}

\newtheorem{rem}[theo]{\rm Remark}

\begin{document}

\title[Localization of Dirac Operators]
{Localization of Dirac Operators on $4n+2$ Dimensional Open spin$^c$ Manifolds}
\author{Shin Hayashi}

\address{Graduate School of Mathematical Sciences, University of Tokyo, 3-8-1 Komaba,
Tokyo, 153-8914, Japan.}

\email{{\tt
hayashi@ms.u-tokyo.ac.jp}}



\begin{abstract}
An integer valued topological index of a Dirac operator is introduced for a pair of a $4n+2$ dimensional open spin$^c$ manifold and a section of the determinant line bundle satisfying some property.
We show a relation between the index and an index of a Dirac operator of its characteristic submanifold, by a localization of a $K$-class. 
\end{abstract}

\maketitle

\tableofcontents

\section{Introduction}

C. Callias studied an index of a Dirac operator on an euclidean space by adding some hermitian operator \cite{Ca78}.
On the other hand, E. Witten gives the localization argument of the de Rham operator $d + d^{\ast}$ \cite{Wi82}, 
both of which are closely related.

A relation between an index of a Dirac operator on a closed spin$^c$ manifold and that of naturally induced spin structure on its characteristic submanifold had been studied by W. Zhang, J. Fast, S. Ochanine, M. Furuta, Y. Kametani \cite{Zh93}, \cite{FO04}, \cite{FK00}.

In this note, we define an integer valued {\em topological index} of a Dirac operator on a $4n+2$ dimensional spin$^c$ manifold, which is not necessarily closed, when a compactly supported section of a determinant line bundle is given (Definition \ref{teigi}), and prove a formula between the index and an index of its characteristic submanifold (Theorem \ref{bb} and Theorem \ref{bc}).

Fast and Ochanine proved the formula for an even dimensional closed spin$^c$ manifold \cite{FO04}.
Their method is to embed the manifold into a sufficiently large euclidean space, and to use the fact that the Thom class of its normal bundle localizes to a neighborhood of its characteristic submanifold when we forget the complex structure and regard it as a $KO$-class. 
Furuta and Kametani, in \cite{FK00}, proved the formula for a $4n$ dimensional compact spin$^c$ manifold with group action and spin boundary.

In contrast, we prove the formula for a $4n+2$ dimensional spin$^c$ manifold which is not necessarily closed (subsection $4.1$).
Our result can be also applied for a $4n+2$ dimensional spin$^c$ manifold with spin boundary.
In this sense, our result complements the result of Furuta and Kametani \cite{FK00}.
It would be possible to generalize Fast and Ochanine's result to open and equivariant cases. 
In this note, however, we give an alternating proof. 
Our result treats only $4n+2$ dimensional spin$^c$ manifolds, though we does not use the embedding and give the proof by localizing the Dirac operator on the manifold and not by expanding the $K$-class. 

Form the perspective of localization, it is desirable to give the formulation of our index not only on closed spin$^c$ manifolds but also on open spin$^c$ manifolds.
We give in section $3$ a definition of an index of a Dirac operator on a $4n+2$ dimensional spin$^c$ manifold, which is not necessarily closed, with the structure we mentioned above.
The key observation is that a section of the determinant line bundle localizes the Dirac operator. 
From this, the sufficient condition for the elliptic boundary condition can be given by using a section of the determinant line bundle.
The additional structure we introduce for the definition of our index gives the elliptic boundary condition.
Our index is twice as the classical one when a spin$^c$ manifold is closed (Remark \ref{3023}).

We also give a short proof of the formula for a $4n+2$ dimensional closed spin$^c$ manifold by using the index formula (subsection $4.2$).
This proof is simple, however it is difficult to prove, in this way, the case when the manifold is open, or the manifold has a group action.
On the other hand, the localization method can easily be generalized to the both cases. 

It will be an interesting problem to show the formula as a relation between $K$ and $KO$-theory, like Fast and Ochanine did, by a localization on the manifold, which does not use an embedding.

This paper is organized as follows.
In section $2$ we summarize some facts about Clifford algebras and its representations.
In section $3$ we give a definition of an index of a Dirac operator on a $4n+2$ dimensional spin$^c$ manifold which is not necessarily closed when some additional structure is given.
We prove in this section that this additional structure can be given by a section of its determinant line bundle.
In section $4$ we prove the formula. In subsection $4.1$ we give the proof by $K$-theory, and in subsection $4.2$ by an index formula.

\section{Preliminaries}
In this section, some facts of Clifford algebras and its representations are summarized.
In this note, we regard real and quaternionic structures (sometimes complex structures also) as, in some sense, one of the structures of Clifford algebras.
Note that the structures we will work on have a period 8. This is because of the periodicity of real Clifford algebras.

We denote $\mathit{Cl}_{n,m}$ as an indefinite Clifford algebra\footnote[1]{We call this algebra simply as {\em Clifford algebra} and permits the adjective {\em indefinite}.}
generated by $n+m$ elements $e_1, \ldots , e_n$, whose squares are $-1$ and $\epsilon_1, \ldots , \epsilon_m$, whose squares are $1$.
Then we define $\mathbb{C}\mathit{l}_n := \mathit{Cl}_{n,0} \otimes \mathbb{C}$ to be a complex Clifford algebra.
Let $S_{2n}$ be a representation space of the irreducible representation of $\mathbb{C}\mathit{l}_{2n}$, 
and $\rho_{2n,0}$ be the representation of $\mathit{Cl}_{2n,0}$ on $S_{2n}$.
We consider a spin$^c$ group $Spin^c_{2n}$ as a subgroup of $\mathbb{C}\mathit{l}_{2n}$ and
denote by $\Delta^{Spin^c}$ the spinor representation of $Spin^c_{2n}$ on $S_{2n}$.
Note that for $\lambda$ $=$ $[\mu, u]$ $\in Spin^c_{2n}$ $\bigl(\mu \in Spin_{2n}, u \in U(1)\bigl)$, 
the relation $\Delta^{Spin^c}(\lambda) = \rho_{2n,0}(\mu)u$ holds. 
For a $\mathbb{Z}_2$-graded vector space $V=V^0 \oplus V^1$, we write $(-1)^{\mathrm{deg}}$ for its $\mathbb{Z}_2$-grading, i.e., $(-1)^{\mathrm{deg}}$ is the operator which is 1 on $V^0$ and $-1$ on $V^1$.
A $\mathbb{Z}_2$-grading of $S_{2n}$ is given by the action of $i^n \rho_{2n,0}(e_1 \cdots e_{2n})$.
We denote $\hat{\otimes}$ for $\mathbb{Z}_2$-graded tensor products, and $\otimes$ for ungraded tensor products.

The following facts of $S_{8n}$ and $S_{8n+4}$ are well known (c.f. \cite{Fr00}).
We will use them to study properties of $S_{8n+2}$ and $S_{8n+6}$.
\begin{lem}
\label{202}
\begin{enumerate}
\item $S_{8n}$ has a real structure  $J \in \mathrm{End}_{\mathbb{R}}(S_{8n})$ which anti-commutes with Clifford multiplication, 
	i.e., anti-linear and $J^2=1$.
\item $S_{8n+4}$ has a quaternionic structure  $J \in \mathrm{End}_{\mathbb{R}}(S_{8n+4})$ which anti-commutes with Clifford multiplication, 
	i.e., anti-linear and $J^2=-1$.
\end{enumerate}
\noindent
In both cases $J$ commutes with $(-1)^{\mathrm{deg}}$ and is an even degree function.
\end{lem}

Now we set some notations.
In $8n+2$ dimensional and $8n+6$ dimensional cases, the following constructions are almost the same, but a little bit different.
This difference derives from the difference between real and quaternionic structures.
In order to treat each case in a same line, we use letters $\mathit{Cl}_0$, $\mathit{Cl}_1$, $\mathit{Cl}_2$, $\mathit{Cl}_3$, $S$, $\rho$ and  $\tau$  for 
$\mathit{Cl}_{8n+2,0}$, $\mathit{Cl}_{8n+2,1}$, $\mathit{Cl}_{8n+2,2}$, $\mathit{Cl}_{8n+2,3}$, $S_{8n+2}$, $\rho_{8n+2,0}$ and $0$ in the $8n+2$ dimensional case,
and for
$\mathit{Cl}_{8n+6,0}$, $\mathit{Cl}_{8n+6,1}$, $\mathit{Cl}_{8n+7,1}$, $\mathit{Cl}_{8n+8,1}$, $S_{8n+6}$, 
$\rho_{8n+6,0}$ and $1$ in the $8n+6$ dimensional case (see the table below).
\vspace{-3mm}
\begin{table}[h]
  \begin{tabular}{|c||c|c|c|c|c|c|c|}
\hline
dimension & $\mathit{Cl}_0$ & $\mathit{Cl}_1$ & $\mathit{Cl}_2$ & $\mathit{Cl}_3$ & $S$ & $\rho$ & $\tau$ \\ \hline \hline
$8n+2$ & $\mathit{Cl}_{8n+2,0}$ & $\mathit{Cl}_{8n+2,1}$ & $\mathit{Cl}_{8n+2,2}$ & $\mathit{Cl}_{8n+2,3}$ & $S_{8n+2}$ &  $\rho_{8n+2,0}$ & $0$\\ \hline
$8n+6$ & $\mathit{Cl}_{8n+6,0}$ & $\mathit{Cl}_{8n+6,1}$ & $\mathit{Cl}_{8n+7,1}$ & $\mathit{Cl}_{8n+8,1}$ & $S_{8n+6}$ &  $\rho_{8n+6,0}$ & $1$\\ \hline
  \end{tabular}
\end{table}
\vspace{-3mm}
\begin{lem}
There is an isomorphism of algebras
\begin{equation*}
\label{last}
\mathit{Cl}_{8n+2+4\tau,1} \cong \mathit{Cl}_{8n+4\tau,0} \underset{\mathbb{R}}{\otimes} \mathit{Cl}_{2,1}.
\end{equation*}
\end{lem}
\vspace{-3mm}
\begin{proof}
We write $e_1, \ldots, e_{8n+2+4\tau}, \epsilon_1$ for the generators of $\mathit{Cl}_{8n+2+4\tau,1}$, $e'_1, \ldots, e'_{8n+4\tau}$ for $\mathit{Cl}_{8n+4\tau,0}$, and $e''_1, e''_2, \epsilon''_1$ for $\mathit{Cl}_{2,1}$. The isomorphism is given by mapping $e_i$ to $e'_i \otimes 1$ $(i=1, \ldots, 8n+4\tau)$, $e_{8n+4\tau+j}$ to $(-1)^{\mathrm{deg}}\otimes e''_j$ $(j=1,2)$, and $\epsilon_1$ to $(-1)^{\mathrm{deg}}\otimes \epsilon''_1$.
\end{proof}
\vspace{-2mm}
\begin{prop}
\label{d}
There is a real representation $\tilde{\rho}$ of $\mathit{Cl}_3$ on $S$ which restricts to the complex representation $\rho$ of $\mathit{Cl}_0$ on $S$.
\end{prop}
\begin{proof}
The proof is divided into two steps.

\noindent
\underline{\bf{Step 1}.}
Let $V = \mathbb{C}=\mathbb{R}^2$ with standard hermitian metric.
We construct a real representation $\rho_{2,3}$ of $\mathit{Cl}_{2,3}$ on $\Lambda_{\mathbb{R}}^{\ast}V$ which restricts to $\rho_{2,0}$.
Let us take a canonical bases $e_1$, $e_2$ of $V=\mathbb{R}^2$, then $e_1$ and $e_2$ generates the algebra $\mathit{Cl}_{2,0}$.
Let $e_1$, $e_2$, $\epsilon_1$, $\epsilon_2$, $\epsilon_3$ be the generators of $\mathit{Cl}_{2,3}$.
We define a complex structure of $\Lambda_{\mathbb{R}}^{\ast}V$ and the representation $\rho_{2,3}$ of $\mathit{Cl}_{2,3}$ on $\Lambda_{\mathbb{R}}^{\ast}V$ by
$$
\left\{
\begin{aligned}
\rho_{2,3}(e_k) &= e_k^{\wedge} - e_k^{\lrcorner} \quad (k=1, 2),\\
\rho_{2,3}(\epsilon_1) &= (-1)^{\mathrm{deg}},\\
\rho_{2,3}(\epsilon_k) &= e_{k-1}^{\wedge} + e_{k-1}^{\lrcorner} \quad (k=2, 3),\\
i \hspace{5mm} &= -\rho_{2,3}(e_1 e_2 \epsilon_1),
\end{aligned}
\right.
$$
where the interior product is defined by the metric naturally induced by the hermitian metric.
Note that 
$\rho_{2,3}(e_k)$, $\rho_{2,3}(\epsilon_1)$ are $\mathbb{C}$-linear, 
and $\rho_{2,3}(\epsilon_k)$ are anti-linear,
and $\rho_{2,3}$ restricts to $\rho_{2,0}$.

\noindent
\underline{\bf{Step 2}.}
Next we consider the other cases.
Under the isomorphism of Lemma \ref{last}, the representation of $\mathit{Cl}_{8n+2+4\tau,1}$
on $S \cong S_{8n+4\tau} \hat{\otimes} S_{2}$ is given by $\rho_0 \otimes 1 + (-1)^{\mathrm{deg}} \otimes \rho_{2,1}$, and actions of the other two generators are given by
$J(-1)^{\mathrm{deg}} \otimes \rho_{2,3}(\epsilon_2)$ and $ J(-1)^{\mathrm{deg}} \otimes \rho_{2,3}(\epsilon_3) 
	\in
		\mathrm{End}_{\mathbb{R}} ( S_{8n+4\tau}\hat{\otimes} S_{2} )$, 
and thus we have the representation.
\end{proof}
\begin{defi}
\label{gg}
A representation $\mathrm{Ad}\otimes z^2$ of $Spin_{8n+2+4\tau}^c$ on $\mathit{Cl}_3$
is defined as follows.
For any $\lambda 
		= 
			[ \mu , u ] \in Spin_{8n+2+4\tau}^c 
		 ( \mu \in Spin_{8n+2+4\tau} ,  u \in U(1))$
and generators $e_1, \ldots, e_{8n+2+4\tau}$, $\epsilon_1$, $\eta_1$, $\eta_2$ of $\mathit{Cl}_3$
(where $\eta_k^2=(-1)^{\tau+1}$),
we set
\begin{equation*}
\begin{cases}
	\mathrm{Ad}\otimes z^2(\lambda, e_i)= \mu e_i \mu^{-1}, &\\
	\mathrm{Ad}\otimes z^2(\lambda, \epsilon_1) =\mu \epsilon_1 \mu^{-1} =\epsilon_1, &\\
	\mathrm{Ad}\otimes z^2(\lambda, \eta_k) = u^2 \eta_k \quad (k=1, 2),& \\
\end{cases}
\end{equation*}
where we define an action of $i$ on $\eta_1$ and $\eta_2$ by 
$i \cdot \eta_1 := \eta_2$ and $i \cdot \eta_2 := -\eta_1$.
\end{defi}
The following two lemmas can easily be checked.
\begin{lem}
\label{mmm}
For any $\lambda \in Spin^{c}_{8n+2+4\tau}$, the following diagram commutes.
\[\xymatrix{
\mathit{Cl}_3 \times S 
	\ar[rr]^(.65){\tilde{\rho}}
	\ar[d]_{(\mathrm{Ad}\otimes z^2)(\lambda) \times \Delta^{Spin^c}(\lambda)} 
	 && 
S 
	\ar[d]^{\Delta^{Spin^c}(\lambda)} \\
\mathit{Cl}_3 \times S 
	\ar[rr]^(.65){\tilde{\rho}}
	&& 
		S \\
}\]
\end{lem}
\begin{lem}
\label{dd} 
\begin{align*}
 \bigl\{\hspace{1mm} v \in \mathit{Cl}_2 \hspace{1mm}|\hspace{1mm} v\epsilon_1 + \epsilon_1 v = 0 ,\hspace{1mm}
ve_i + e_iv = 0 \hspace{1mm} (i=1 \ldots 8n+2+4\tau \hspace{1mm})\bigl\}\\
	=
		 \bigl\{ \hspace{1mm}a\epsilon_2 + b e_1 e_2 \cdots e_{8n+2+4\tau} \epsilon_1 \epsilon_2 \in \mathit{Cl}_2 \hspace{1mm}|\hspace{1mm} a, b \in \mathbb{R} \hspace{1mm}\bigl\}.
\end{align*}
\end{lem}
\begin{rem}
\label{ddd}
Note that there is an isomorphism of algebras.
\begin{equation*}
\mathrm{End}_{\mathbb{R}}(S) \cong \mathbb{R}(4^{\tau}\cdot4\cdot16^n) \cong \mathit{Cl}_2.
\end{equation*}
\end{rem}

\section{An Index on a $4n+2$ Dimensional Open Spin$^c$ Manifold}
Let $X$ be a $4n+2$ dimensional spin$^c$ manifold which is not necessarily closed.
We denote by $\pi \colon TX \rightarrow X$ the tangent bundle and by $Cl(X)$ the Clifford bundle of $TX$ with its riemannian metric.
Let $\mathit{P}_{Spin^c}(X)$ be the principal $Spin^c_{4n+2}$ bundle, 
and we write its transition functions as $\{ h_{\alpha \beta} \}$.
$ h_{\alpha \beta} $ can be written as $h_{\alpha \beta} = [ g_{\alpha \beta}, f_{\alpha \beta} ]$ using functions 
$g_{\alpha \beta} \colon U_{\alpha} \cap U_{\beta} \rightarrow Spin_{4n+2}$ and
$f_{\alpha \beta} \colon U_{\alpha} \cap U_{\beta} \rightarrow U(1)$.
Let $S(X)$ denote the spinor bundle of $X$
and write $c_{X}$ for the Clifford action of $Cl(X)$ on $S(X)$.
We denote by $\mathcal{L}_X$ the determinant line bundle of the spin$^c$ structure.

We take a section $h_{\mathrm{det}}$ of $\mathcal{L}_X$ and assume that $h_{\mathrm{det}}^{-1}(0) \subset X$ is compact.
Let $E = E^0 \oplus E^1 \rightarrow X$ be a $\mathbb{Z}_2$-graded hermitian vector bundle which has a structure of a complexification of some real vector bundle on $X-K$, the complement of a compact set $K$ in $X$\footnote[2]{A complex vector bundle with this structure is called to have a {\em Lagrangian subbundle}.} .
In this section we construct an element of $K_{\mathrm{cpt}}(TX)$ from these data and define an index.

Since $X$ can be noncompact, we adopt a principle of localizing a Dirac operator on a neighborhood of some compact set in order to define an element of $K_{\mathrm{cpt}}(TX)$.
We now consider the following subbundle of $\mathrm{End}_{\mathbb{R}}(S(X))$ in order to formalize a structure which localizes a Dirac operator.

\begin{defi}
\label{3010}
Let $\mathcal{L}_{\mathcal{H}}$ be a subbundle of  $\mathrm{End}_{\mathbb{R}}(S(X))$ consists of the elements which is 
(skew-)symmetric, odd degree, and anti-commute with Clifford multiplication.
We consider symmetric elements in $8n+2$ dimensional cases, and skew-symmetric elements in $8n+6$ dimensional cases.
\end{defi}

\begin{lem}
\label{line}
$\mathcal{L}_\mathcal{H}$ is a complex line bundle.
\end{lem}

\begin{proof}
Since $v \in \mathcal{L}_{\mathcal{H}}$ anti-commutes with $(-1)^{\tau}i \rho(e_1 \cdots e_{8n+2+4\tau})$ and Clifford multiplication, $v$ is anti-linear.
It follows easily that $iv$ belongs to $\mathcal{L}_{\mathcal{H}}$.
The lemma then follows from Lemma \ref{dd} and Remark \ref{ddd}.
\end{proof}
\begin{lem}
\label{k}
For any section $h_{\mathrm{End}}$ of $\mathcal{L}_{\mathcal{H}}$ whose support\footnote[3]{
For an endomorphism $f\colon E \rightarrow E$ of a vector bundle $E \rightarrow X$, we call its {\em support} as the subset of $X$ consists of the elements on which $f$ is not invertible, and denote by supp$(f)$.} is compact, the support of
$\sigma_{X, h_{\mathrm{End}}} :=
		c_{X} \otimes i + \pi^{\ast}h_{\mathrm{End}} \otimes  i^{\tau} $
is also compact.
\end{lem}
\begin{proof}
It is enough to show that $\mathrm{supp}(\sigma_{X,h_{\mathrm{End}}}) \cong \mathrm{supp}(h_{\mathrm{End}})$.
Since $h_{\mathrm{End}}$ anti-commutes with Clifford multiplication, we have
 \begin{align*}
    (c_{X} \otimes i+\pi^{\ast}h_{\mathrm{End}} \otimes  i^{\tau})^2 
      &=\{ -c_{X}^2 +(-1)^{\tau}(\pi^{\ast}h_{\mathrm{End}})^2\} \otimes 1.
\end{align*}
Since $-c_{X}^2$ and $(-1)^{\tau}(\pi^{\ast}h_{\mathrm{End}})^2$ are both positive, we have
\begin{equation*}
\mathrm{supp}(\sigma_{X, h_{\mathrm{End}}}) = \mathrm{supp}(c_X) \cap \mathrm{supp}(\pi^{\ast}h_{\mathrm{End}}) \cong \mathrm{supp}(h_{\mathrm{End}}).
\end{equation*}
\end{proof}
\vspace{-5mm}
\begin{defi}
\label{gen}
We can define an element of $K_{\mathrm{cpt}}(TX)$ by
\begin{equation*}
\alpha(h_{\mathrm{End}}) := [\pi^{\ast}S(X)^0 \underset{\mathbb{R}}{\otimes}\mathbb{C},
	 \pi^{\ast}S(X)^1\underset{\mathbb{R}}{\otimes}\mathbb{C}\colon 
		\sigma_{X, h_{\mathrm{End}}}]
	\in	K_{\mathrm{cpt}}(TX).
\end{equation*}
\end{defi}
\vspace{-1mm}

We can generalize the construction above to twisted cases.
In the setting of Lemma \ref{k}, we assume further that we are given a $\mathbb{Z}_2$-graded hermitian vector bundle $E \rightarrow X$ 
which has a structure of a complexification of some real vector bundle on the complement of some compact set $K$ in $X$.
This complexification structure gives a complex conjugation on $X-K$, i.e., a continuous section $s$ of $\mathrm{End}_{\mathbb{R}}(E)|_{X-K}$ such that, at each $x \in X-K$, $s(x)$ is even degree, anti-linear, and whose square is $1$.
Let us choose an open neighborhood of $K$ in $X$ which is homotopic to $K$.
We take a real valued continuous function $f$ on $X$ whose value is $0$ on $K$, $1$ on the complement of $U$ and non-zero on $U-K$.
A product $h' := fs$ defines a global section of $\mathrm{End}_{\mathbb{R}}(E)$.
If we set $h'_{\mathrm{End}} := h_{\mathrm{End}} \otimes h' \in \mathrm{End}_{\mathbb{R}}(S(X) \hat{\otimes} E)$ then, 
at each point $x \in X$, $h'_{\mathrm{End}}(x)$ is (skew-)symmetric, odd degree, anti-commutes with Clifford multiplication, and its support is compact.
In the same manner as above, we can define an element of $K_{\mathrm{cpt}}(TX)$ by
\begin{equation*}
\alpha(E, h_{\mathrm{End}}) := [\hspace{1mm} \pi^{\ast} \bigl( S(X) \underset{\mathbb{C}}{\hat{\otimes}} E \bigl)^0\underset{\mathbb{R}}{\otimes}\mathbb{C}, 
				\hspace{1mm} \pi^{\ast}\bigl( S(X)\underset{\mathbb{C}}{\hat{\otimes}} E \bigl)^1\underset{\mathbb{R}}{\otimes}\mathbb{C} \hspace{1mm}\colon\hspace{1mm} 
					\sigma_{X,E,h_{\mathrm{End}}}]
	\in K_{\mathrm{cpt}}(TX),
\end{equation*}
where 
$\sigma_{X,E,h_{\mathrm{End}}}
	:=
		\bigl( c_{X}\otimes1 \bigl) \otimes i + \pi^{\ast}h'_{\mathrm{End}} \otimes i^{\tau}$.
\begin{rem}
A tensor product of two anti-linear maps defines an anti-linear map between a $\mathbb{C}$-tensor product.
The complex conjugation of $E$ is necessary because $h_{\mathrm{End}}$ is anti-linear, and we need another anti-linear endomorphism of $E$ in order to get a well-defined endomorphism of the $\mathbb{C}$-tensor product of $S(X)$ and $E$.
\end{rem}

\begin{rem}
Let $V$ be a neighborhood of $\mathrm{supp}(\sigma_{X,E,h_{\mathrm{End}}})$ in $TX$.
We can ignore the information of $\alpha(E, h_{\mathrm{End}})$ outside $V$ since $\sigma_{X,E,h_{\mathrm{End}}}$ gives isomorphism there, 
and we can consider $\alpha(E, h_{\mathrm{End}})$ as an element of $K_{\mathrm{cpt}}(V)$.
This is because the symbol of the Dirac operator is localized near $\mathrm{supp}(\sigma_{X,E,h_{\mathrm{End}}})$ by the perturbation by $\pi^{\ast}h'_{\mathrm{End}} \otimes i^{\tau}$.
In this note we call this as a {\em localization} of the Dirac operator.
\end{rem}
We have thus obtained an element of $K_{\mathrm{cpt}}(TX)$ from a compactly supported section of $\mathcal{L}_{\mathcal{H}}$.
We next show that a compactly supported section of $\mathcal{L}_X$ gives an element of $K_{\mathrm{cpt}}(TX)$.
\begin{defi}
\label{3017}
We define an {\em indefinite Clifford bundle} on $X$ by
\begin{equation*}
	\widetilde{\mathit{Cl}}(X)
		:= 
			\mathit{P}_{Spin^c}(X)
				\underset{\mathrm{Ad}\otimes z^2}{\times}
					\mathit{Cl}_3.
\end{equation*}
\end{defi}
\vspace{-2mm}
Definition \ref{gg} shows that we have a subbundle of $\widetilde{\mathit{Cl}}(X)$ spanned by $\eta_1, \eta_2$ at each point, which is a complex line bundle.
We denote it by $\mathcal{L}_{\mathit{Cl}}$.
\begin{lem}
\label{adad}
$\mathcal{L}_{\mathit{Cl}} \cong \mathcal{L}_{\mathit{X}}$.
\end{lem}
\begin{proof}
Since $\mathcal{L}_X = P_{Spin^c} \underset{z^2}{\times} \mathbb{C}$, the lemma follows from Definition \ref{gg}.
\end{proof}
\vspace{-1mm}
If we take $\lambda = h_{\alpha \beta}(x)$ at Lemma \ref{mmm}, the next result follows.
\begin{prop}
\label{ad}
The action $\tilde{\rho}$ of $\mathit{Cl}_3$ on $S$ gives the representation $\tilde{c}_X$ of $\widetilde{\mathit{Cl}}(X)$ on $S(X)$, 
and thus we have a bundle map 
$\tilde{c}_X \colon \mathcal{L}_X \rightarrow \mathrm{End}_{\mathbb{R}}(S(X))$.
\end{prop}
\begin{lem}
\label{3016}
	$\mathcal{L}_X \cong \mathcal{L}_{\mathcal{H}}$.
\end{lem}
\begin{proof}
From the construction of $\tilde{\rho}$, it can easily be checked that $\tilde{c}_X$ maps $\mathcal{L}_{\mathit{Cl}}$ to $\mathcal{L}_{\mathcal{H}}$.
Since both $\mathcal{L}_{\mathcal{H}}$ and $\mathcal{L}_{\mathit{Cl}}$ are complex line bundles and $\tilde{c}_X$ is nonzero bundle map, $\tilde{c}_X$ gives the isomorphism between $\mathcal{L}_{\mathcal{H}}$ and  $\mathcal{L}_{\mathit{Cl}}$.
Then the result follows form Lemma \ref{adad}.
\end{proof}
Since $\tilde{c}_X$ is given by the Clifford action, the following theorem is proved.
\begin{theo}
\label{ss}
There is a one-to-one correspondence between the sections of $\mathcal{L}_X$ and the sections of $\mathcal{L}_\mathcal{H}$, 
which preserves supports.
\end{theo}

We are now in a position to define an index of a $4n+2$ dimensional spin$^c$ manifold which is not necessarily closed.
\begin{defi}
\label{teigi}
Let $X$ be a $4n+2$ dimensional spin$^c$ manifold, and $S(X)$ be its spinor bundle.
Let $E \rightarrow X$ be a $\mathbb{Z}_2$-graded hermitian vector bundle which has a structure of a complexification of some real bundle on a complement of some compact set.
Let $h_{\mathrm{det}}$ be a section of $\mathcal{L}_X$ which has a compact support.
We take $h_{\mathrm{End}}=\tilde{c}_X \circ h_{\mathrm{det}}$, which is the section of $\mathcal{L}_\mathcal{H}$ corresponds to $h_{\mathrm{det}}$ by Theorem \ref{ss}.
Then we have an element $\alpha(h_{\mathrm{End}})$ of $K_{\mathrm{cpt}}(TX)$ constructed by the Definition \ref{gen}.
We call an integer given by sending $\alpha(h_{\mathrm{End}})$ to $\mathbb{Z}$ by an index map as an {\em index} of the triple $(X, E, h_{\mathrm{det}})$ and denote by $\mathrm{index}_{\mathbb{R}}(X, E, h_{\mathrm{det}})$.

We also define an indiex of a pair $(X, h_{\mathrm{det}})$ as in the same way, and denote by $\mathrm{index}_{\mathbb{R}}(X, h_{\mathrm{det}})$.
\end{defi}
\begin{rem}
\label{3023}
If $X$ is closed, since $h_{\mathrm{End}}$ is homotopic to the zero section in the space of compactly supported sections of $\mathcal{L}_{\mathcal{H}}$, the element $\alpha(h_{\mathrm{End}})$ in $K_{\mathrm{cpt}}(TX)$ does not depend on the choice of $h_{\mathrm{End}}$.
If we denote by $\mathrm{index}_{\mathbb{C}}(X)$ the index of the Dirac operator of the spin$^c$ structure of $X$, 
then the relation between $\mathrm{index}_{\mathbb{R}}(X, h_{\mathrm{det}})$ and $\mathrm{index}_{\mathbb{C}}(X)$ is given by
\begin{equation*}
\mathrm{index}_{\mathbb{R}}(X, h_{\mathrm{det}}) = \mathrm{index}_{\mathbb{R}}(X, 0) = 2 \cdot \mathrm{index}_{\mathbb{C}}(X),
\end{equation*}
where multiplication by $2$ is derived from the complexification at the construction of $\mathrm{index}_{\mathbb{R}}(X, h_{\mathrm{det}})$.
In this sense, our definition of the index on a $4n+2$ dimensional spin$^c$ manifold with some extra structure gives the generalization of the classical index on a closed spin$^c$ manifold.
\end{rem}

\section{Relation with an Index of a Characteristic Submanifold}

In this section, we prove a formula between the index of a $4n+2$ dimensional spin$^c$ manifold and that of its characteristic submanifold.
We give a proof in two ways, first by the localization of a $K$-class (subsection $4.1$), second by the index formula (subsection $4.2$).

\subsection{Proof by $K$-theory}

In the setting of Definition \ref{teigi}, 
we assume further that $h_{\mathrm{det}}$ intersects transversely to the zero section. 
We have a submanifold $Y$$:=$ $h_{\mathrm{det}}^{-1}(0)$ in $X$, which is known as a characteristic submanifold.
The characteristic submanifold $Y$ has a spin$^c$ structure naturally induced by the spin$^c$ structure of $X$.
This spin$^c$ structure of $Y$ actually is a spin structure because whose determinant line bundle is trivial, so $Y$ is a spin manifold.

In this subsection we prove the following main theorem of this note.
\begin{theo}
\label{bb}
Let $X$ be a $4n+2$ dimensional spin$^c$ manifold.
Suppose that we are given a section $h_{\mathrm{det}}$ of $\mathcal{L}_X$, which intersects transversely to the zero section and $h_{\mathrm{det}}^{-1}(0)$ is compact. 
$Y=h_{\mathrm{det}}^{-1}(0)$ is a characteristic submanifold of $X$.
Then we have the following equation between the index of the pair $(X, h_{\mathrm{det}})$ and the index of the Dirac operator of $Y$.
\vspace{-1mm}
\begin{equation*}
	\mathrm{index}_{\mathbb{R}}(X,h_{\mathrm{det}}) = \mathrm{index}_{\mathbb{C}}(Y)
\end{equation*}
\end{theo}
\vspace{-1mm}
The twisted case is as follows.
\vspace{-1mm}
\begin{theo}
\label{bc}
Let $X$ be a $4n+2$ dimensional spin$^c$ manifold.
Suppose that we are given a section $h_{\mathrm{det}}$ of $\mathcal{L}_X$, which intersects transversely to the zero section and $h_{\mathrm{det}}^{-1}(0)$ is compact.
$Y=h_{\mathrm{det}}^{-1}(0)$ is a characteristic submanifold of $X$.
Let $E \rightarrow X$ be a $\mathbb{Z}_2$-graded hermitian vector bundle which has a structure of a complexification of a real vector bundle on $X$.
Then we have the following equation between the index of the triple $(X, E, h_{\mathrm{det}})$ and the index of the Dirac operator of $Y$ twisted by $E|_Y$. 
\begin{equation*}
	\mathrm{index}_{\mathbb{R}}(X, E, h_{\mathrm{det}}) = \mathrm{index}_{\mathbb{C}}(Y, E|_{Y}).
\end{equation*}
\end{theo}
\begin{rem}
\label{44}
If $X$ is closed, by Remark \ref{3023}, we have
$
2 \cdot \mathrm{index}_{\mathbb{C}}(X) = \mathrm{index}_{\mathbb{C}}(Y).
$
\end{rem}
The tubular neighborhood theorem says that we can take a tubular neighborhood $N$ of $Y$ in $X$ such that $N \cong \mathcal{L}_X|_Y$.
To avoid confusion, we denote by $\pi \colon \nu \rightarrow Y$ when we consider $N$ as a complex line bundle over $Y$, so we have $\nu \cong N \cong \mathcal{L}_X|_Y$.
We denote by $\tilde{\pi} \colon TN \rightarrow TY$ for the bundle projection induced by $\pi$, and denote by
$\pi_3 \colon TX \rightarrow X$, 
$\pi_2 \colon TN \rightarrow N$, and
$\pi_1 \colon TY \rightarrow Y$ 
for projections of tangent bundles.
\[\xymatrix{
TY \ar[d]_{\pi_1}& TN \ \ar[d]^{\pi_2} \ar[l]_{\tilde{\pi}} \ar @{^{(}->}[r] \ar[ld]_{\hat{\pi}}& TX \ar[d]^{\pi_3}\\
Y & N \ \ar[l]^{\pi} \ar @{^{(}->}[r] & X 
}\]

For the bundle $TN \rightarrow TY$ there is an isomorphism $TN \cong \pi_1^{\ast}\nu \oplus \pi_1^{\ast}\nu$.
We regard left $\pi_1^{\ast}\nu$ as a base direction component, and right $\pi_1^{\ast}\nu$ a fiber direction component.
We express one point $u \in TN$ by four components $u=(y, \xi, u_b, u_f)$, 
where $y$ is an element of $Y$, 
$\xi$ is a fiber component of $\pi_1 \colon TY \rightarrow Y$, 
$u_b$ is a base direction element of a fiber of $\tilde{\pi} \colon TN \rightarrow TY$, and
$u_f$ is a fiber direction element of a fiber of $\tilde{\pi} \colon TN \rightarrow TY$.
We define a complex structure on $TN$ by $i \cdot (u_b, u_f) := (-u_f, u_b)$.
In this sense, $TN$ can be seen as a complexification $TN \cong \pi_1^{\ast}\nu \otimes \mathbb{C}$, and 
$\tilde{\pi} \colon TN \rightarrow TY$ thus has a complex vector bundle structure.

Let $\lambda_{TN}$ be the Thom class of the complex vector bundle $TN \rightarrow TY$.
In order to prove Theorem \ref{bb}, it is enough to show $\alpha(h_{\mathrm{End}})|_N=v\cdot \lambda_{TN}$ 
as an element of $K_{\mathrm{cpt}}(TN)$ since we have $\mathrm{index}_{\mathbb{C}}(Y) = \mathrm{index}_{\mathbb{C}}(v) = \mathrm{index}_{\mathbb{C}}(v \cdot \lambda_{TN})$ for  $v=[\pi_1^{\ast}S^0(Y), \pi_1^{\ast}S^1(Y) \colon i\cdot c_Y]$ $\in$ $K_{\mathrm{cpt}}(TY)$.
Therefore, it is enough to show that the following two elements in $K_{\mathrm{cpt}}(TN)$ coincide.
\vspace{-2mm}
\[\xymatrix{
\pi_2^{\ast}S(X)|_N \underset{\mathbb{R}}{\otimes} \mathbb{C}
\ar[rrrrr]^{c_X \otimes i + \pi_2^{\ast}h_{\mathrm{End}} \otimes i^{\tau}}
&&&&&
\pi_2^{\ast}S(X)|_N \underset{\mathbb{R}}{\otimes} \mathbb{C}
}\]
\vspace{-4mm}
\[\xymatrix{
\hat{\pi}^{\ast}S(Y) 
	\underset{\mathbb{C}}{\hat{\otimes}}
		\tilde{\pi}^{\ast}\Lambda_{\mathbb{C}}^{\ast}TN
\ar[rrrr]^{i\tilde{\pi}^{\ast}c_Y \otimes 1 + (-1)^{\mathrm{deg}} \otimes (\wedge + \lrcorner)}
&&&&
\hat{\pi}^{\ast}S(Y) 
	\underset{\mathbb{C}}{\hat{\otimes}}
		\tilde{\pi}^{\ast}\Lambda_{\mathbb{C}}^{\ast}TN
}\]

The proof is devided into three steps.

\noindent
\underline{\bf{Step 1}.}
We have a relation $S(X)|_Y \cong S(Y) \hat{\otimes} S(\nu) \cong S(Y) \hat{\otimes} \Lambda_{\mathbb{C}}^{\ast}\pi^{\ast}\nu$ between spinors, 
and pull-back by $\pi_2$ gives 
$\pi_2^{\ast}S(X)|_N
	\cong
\hat{\pi}^{\ast}S(Y) \hat{\otimes} \tilde{\pi}^{\ast}\Lambda_{\mathbb{C}}^{\ast}\pi_1^{\ast}\nu$.
On both sides, there is a Clifford action of $\mathit{Cl}(N)$.
Since $\mathit{Cl}(N)_{u_b}$ is generated by the components of $\xi$ and $u_f$,
the relation of these two actions is given by
\vspace{-2mm}
\begin{equation*}
	c_X(u) = \tilde{\pi}^{\ast}c_Y(\xi) \otimes 1 + (-1)^{\mathrm{deg}} \otimes (u_f^{\wedge}- u_f^{\lrcorner}).
\end{equation*}
\vspace{-2mm}
Therefore we have the following commutative diagram.
\[\xymatrix{
\pi_2^{\ast}S(X)|_N \underset{\mathbb{R}}{\otimes} \mathbb{C}
\ar[rrrrr]^{c_X\otimes i}
\ar[d]_{\cong}
	&&&&&
\pi_2^{\ast}S(X)|_N \underset{\mathbb{R}}{\otimes} \mathbb{C}
\ar[d]^{\cong}
	\\
\Bigl( \hat{\pi}^{\ast}S(Y)\underset{\mathbb{C}}{\hat{\otimes}}
	\tilde{\pi}^{\ast}\Lambda_{\mathbb{C}}^{\ast}\pi_1^{\ast}\nu \Bigl) \underset{\mathbb{R}}{\otimes} \mathbb{C}
\ar[rrrrr]^{ \bigl(\tilde{\pi}^{\ast}c_Y \otimes 1 + (-1)^{\mathrm{deg}} \otimes (u_f^{\wedge}- u_f^{\lrcorner})\bigl)\otimes i }
\ar[d]_{\cong}
	&&&&&
\Bigl( \hat{\pi}^{\ast}S(Y)\underset{\mathbb{C}}{\hat{\otimes}}
	\tilde{\pi}^{\ast}\Lambda_{\mathbb{C}}^{\ast}\pi_1^{\ast}\nu \Bigl) \underset{\mathbb{R}}{\otimes} \mathbb{C}
\ar[d]^{\cong}
	\\
\hat{\pi}^{\ast}S(Y)\underset{\mathbb{C}}{\hat{\otimes}}
	\Bigl(\tilde{\pi}^{\ast}\Lambda_{\mathbb{C}}^{\ast}\pi_1^{\ast}\nu \underset{\mathbb{R}}{\otimes} \mathbb{C}\Bigl)
\ar[rrrrr]^{i\tilde{\pi}^{\ast}c_Y \otimes (1 \otimes 1) + (-1)^{\mathrm{deg}}\otimes \bigl( (u_{f}^{\wedge}-u_{f}^{\lrcorner})\otimes i \bigl)}
	&&&&&
\hat{\pi}^{\ast}S(Y)\underset{\mathbb{C}}{\hat{\otimes}}
	\Bigl(\tilde{\pi}^{\ast}\Lambda_{\mathbb{C}}^{\ast}\pi_1^{\ast}\nu \underset{\mathbb{R}}{\otimes} \mathbb{C}\Bigl)
}\]
where the second of the vertical isomorphism is given by
\vspace{-1mm}
\begin{equation*}
(x \underset{\mathbb{C}}{\otimes}y)\underset{\mathbb{R}}{\otimes}z 
	\longmapsto 
\frac{J+1}{2}x \underset{\mathbb{C}}{\otimes} \bigl(y \underset{\mathbb{R}}{\otimes}z \bigl) + \frac{J-1}{2}x \underset{\mathbb{C}}{\otimes} \bigl( iy \underset{\mathbb{R}}{\otimes}iz \bigl).
\end{equation*}

\noindent
\underline{\bf{Step 2}.}
Through the isomorphism $\mathcal{L}_X|_N \cong \pi^{\ast}\nu$, 
$h_{\mathrm{det}}|_N$ corresponds to the tautological section $t$ of $\pi^{\ast}\nu \rightarrow N$, 
i.e., $t(y, u_b) = u_b$.
This follows that, through the isomorphism
$\mathrm{End}_{\mathbb{R}}(S(X)|_N)
	\cong
		\mathrm{End}_{\mathbb{R}}(\pi^{\ast}S(Y)\hat{\otimes} \Lambda_{\mathbb{C}}^{\ast} \pi^{\ast} \nu)$, 
$h_{\mathrm{End}}|_N$ corresponds to $ J(-1)^{\mathrm{deg}} \otimes ( u_b^{\wedge} + u_b^{\lrcorner} )$
(see the proof of Proposition \ref{d}).
Thus we have the following diagram.
\[\xymatrix{
\pi_2^{\ast}S(X)|_N \underset{\mathbb{R}}{\otimes} \mathbb{C}
\ar[rrr]^{\pi_2^{\ast}h_{\mathrm{End}}\otimes i^{\tau}}
\ar[d]_{\cong}
	&&&
\pi_2^{\ast}S(X)|_N \underset{\mathbb{R}}{\otimes} \mathbb{C}
\ar[d]^{\cong}
	\\
\Bigl( \hat{\pi}^{\ast}S(Y)\underset{\mathbb{C}}{\hat{\otimes}}
	\tilde{\pi}^{\ast}\Lambda_{\mathbb{C}}^{\ast}\pi_1^{\ast}\nu \Bigl) \underset{\mathbb{R}}{\otimes} \mathbb{C}
\ar[rrr]^{ \bigl( J(-1)^{\mathrm{deg}} \otimes (u_b^{\wedge} + u_b^{\lrcorner}) \bigl) \otimes i^{\tau} }
\ar[d]_{\cong}
	&&&
\Bigl( \hat{\pi}^{\ast}S(Y)\underset{\mathbb{C}}{\hat{\otimes}}
	\tilde{\pi}^{\ast}\Lambda_{\mathbb{C}}^{\ast}\pi_1^{\ast}\nu \Bigl) \underset{\mathbb{R}}{\otimes} \mathbb{C}
\ar[d]^{\cong}
	\\
\hat{\pi}^{\ast}S(Y)\underset{\mathbb{C}}{\hat{\otimes}}
	\Bigl(\tilde{\pi}^{\ast}\Lambda_{\mathbb{C}}^{\ast}\pi_1^{\ast}\nu \underset{\mathbb{R}}{\otimes} \mathbb{C}\Bigl)
\ar[rrr]^{ (-1)^{\mathrm{deg}} \otimes \bigl( i^{\tau}(u_b^{\wedge} + u_b^{\lrcorner}) \otimes 1 \bigl)}
	&&&
\hat{\pi}^{\ast}S(Y)\underset{\mathbb{C}}{\hat{\otimes}}
	\Bigl(\tilde{\pi}^{\ast}\Lambda_{\mathbb{C}}^{\ast}\pi_1^{\ast}\nu \underset{\mathbb{R}}{\otimes} \mathbb{C}\Bigl)
}\]

\noindent
\underline{\bf{Step 3}}.
By Step $1$ and Step $2$, the following diagram is obtained.
\vspace{-1mm}
\[\xymatrix{
\pi_2^{\ast}S(X)|_N  \underset{\mathbb{R}}{\otimes} \mathbb{C}
\ar[rrrr]^{c_X \otimes i + \pi_2^{\ast}h_{\mathrm{End}} \otimes i^{\tau}}
\ar[d]_{\cong}
	&&&&
\pi_2^{\ast}S(X)|_N  \underset{\mathbb{R}}{\otimes}{\mathbb{C}}
\ar[d]^{\cong} \\
\hat{\pi}^{\ast}S(Y) 
	\underset{\mathbb{C}}{\hat{\otimes}}
		\Bigl( \tilde{\pi}^{\ast} \Lambda_{\mathbb{C}}^{\ast} \pi_1^{\ast} \nu
			\underset{\mathbb{R}}{\otimes}{\mathbb{C}} \Bigl)
\ar[rrrr]^{\theta}
	&&&&
\hat{\pi}^{\ast}S(Y) 
	\underset{\mathbb{C}}{\hat{\otimes}}
		\Bigl( \tilde{\pi}^{\ast} \Lambda_{\mathbb{C}}^{\ast} \pi_1^{\ast} \nu
			\underset{\mathbb{R}}{\otimes}{\mathbb{C}} \Bigl)
}\]
where
$
\theta = i\tilde{\pi}^{\ast}c_Y \otimes (1 \otimes 1) +  (-1)^{\mathrm{deg}} \otimes \Bigl( i^{\tau}(u_b^{\wedge} + u_b^{\lrcorner}) \otimes 1 + (u_f^{\wedge}- u_f^{\lrcorner}) \otimes i\Bigl).
$

\vspace{1mm}
What is left to show is the following lemma.
\vspace{-1mm}
\begin{lem}
\label{416}
There is an isomorphism 
$\tilde{\pi}^{\ast} \Lambda_{\mathbb{C}}^{\ast}\pi_1^{\ast}\nu \underset{\mathbb{R}}{\otimes} \mathbb{C}
\cong
\tilde{\pi}^{\ast} \Lambda_{\mathbb{C}}^{\ast} \Bigl( \pi_1^{\ast} \nu
			\underset{\mathbb{R}}{\otimes}{\mathbb{C}} \Bigl)$ 
which commutes the following diagram.
\vspace{-1mm}
\[\xymatrix{
\tilde{\pi}^{\ast} \Lambda_{\mathbb{C}}^{\ast}\pi_1^{\ast}\nu \underset{\mathbb{R}}{\otimes} \mathbb{C}
\ar[d]_{\cong}
\ar[rrrr]^{i^{\tau}(u_b^{\wedge} + u_b^{\lrcorner}) \otimes 1 + (u_f^{\wedge}- u_f^{\lrcorner}) \otimes i}
	&&&&
\tilde{\pi}^{\ast} \Lambda_{\mathbb{C}}^{\ast}\pi_1^{\ast}\nu \underset{\mathbb{R}}{\otimes} \mathbb{C}
\ar[d]^{\cong}\\
\tilde{\pi}^{\ast} \Lambda_{\mathbb{C}}^{\ast} \Bigl( \pi_1^{\ast} \nu
			\underset{\mathbb{R}}{\otimes}{\mathbb{C}} \Bigl)
\ar[rrrr]^{(\wedge + \lrcorner)(u_b, u_f)}
	&&&&
\tilde{\pi}^{\ast} \Lambda_{\mathbb{C}}^{\ast} \Bigl( \pi_1^{\ast} \nu
			\underset{\mathbb{R}}{\otimes}{\mathbb{C}} \Bigl)
}\]
\end{lem}
\vspace{-3mm}
\begin{proof}
\label{kakan}
Since principal bundles of two vector bundles are the same, 
it is enough to show that, 
for a complex $1$ dimensional vector space $V$ with hermitian metric, 
there is an isomorphism between 
$\Lambda_{\mathbb{C}}^{\ast} V \otimes \mathbb{C}$ and $\Lambda_{\mathbb{C}}^{\ast} \bigl( V \otimes \mathbb{C} \bigl)$ which commutes the corresponding diagram.
Both of the horizontal arrows are complex representations of $\mathit{Cl}(V \oplus V) \cong \mathit{Cl}_{0,4}$, and so are irreducible representations of $\mathit{Cl}(V \oplus V)\otimes \mathbb{C} \cong\mathbb{C}l_4$ since both dimensions of representations are 4.
The existence of the desired isomorphism follows from the uniqueness of the irreducible representation of $\mathbb{C}l_4$.
\end{proof}
\vspace{-4mm}

Thus we proved Theorem \ref{bb}. Theorem \ref{bc} can be proved in the same line, using the real or quaternionic structure of $S(X)\hat{\otimes} E$.
\vspace{-1mm}
\subsection{Proof by the Index Formula}
In this subsection, we prove the following theorem by the index formula.
\begin{theo}
\label{f}
Let $X$ be a $4n+2$ dimensional closed spin$^c$ manifold, and $Y$ be a characteristic submanifold of $X$.
Let $E \rightarrow X$ be a $\mathbb{Z}_2$-graded hermitian vector bundle which has a structure of complexification of some real vector bundle on $X$.
Then we have the following equation.
\begin{equation*}
	2 \cdot \mathrm{index}_{\mathbb{C}}(X, E) = \mathrm{index}_{\mathbb{C}}(Y, E|_Y).
\end{equation*}
\end{theo}
This corresponds to the special case of Theorem \ref{bc} when the manifold $X$ is closed (see Remark \ref{3023}).
The key is the localization of the element of the ordinary cohomology by the use of the Poincar$\acute{\mathrm{e}}$ duality.
\vspace{-1mm}
\begin{proof}
Let $x = c_1(\mathcal{L}_X)=e(\mathcal{L}_X)$, then
\begin{gather*}
\mathrm{index}(X, E) = \bigl< \mathrm{ch}(E) \mathrm{Td}(X) \hspace{1mm},\hspace{1mm} [X] \bigl> 
  = \bigl< \mathrm{ch}(E) \mathrm{ch}(\mathcal{L}_X) \hat{A}(X)\hspace{1mm},\hspace{1mm} [X] \bigl> \\
  = \bigl< \mathrm{ch}(E) e^{x/2}\hat{A}(X) \hspace{1mm},\hspace{1mm} [X] \bigl>,
\end{gather*}
\vspace{-4mm}
\begin{gather*}
\mathrm{index}(Y, E|_Y) = \bigl< \mathrm{ch}(E) \hat{A}(Y) \hspace{1mm},\hspace{1mm} [Y] \bigl> 
 = \Bigl< \mathrm{ch}(E) \hat{A}(\mathcal{L}_X|_Y)^{-1} \hat{A}(TX|_Y), \hspace{1mm} [Y] \Bigl> \\
 =\Bigl< \mathrm{ch}(E) \cdot \Bigl( \cfrac{x}{e^{x/2}-e^{-x/2}} \Bigl)^{-1} \cdot \hspace{1mm} \hat{A}(X) \cdot x \hspace{1mm},\hspace{1mm} [X]\Bigl>
 = \Bigl< \mathrm{ch}(E) (e^{x/2}-e^{-x/2}) \hat{A}(X) \hspace{1mm},\hspace{1mm} [X]\Bigl>.
\end{gather*}

The degree of $[X]$ is $4n+2$, that of $x$ is $2$, and that of $\hat{A}(X)$ is a multiple of $4$.
Since $E \cong \overline{E}$, the degree of $ch(E)$ is also a multiple of $4$.
Note that the value evaluated by the fundamental class $[X]$ depends only on the odd degree term of $x$ and do not depend on the even term.
The equation follows because the twice of the odd degree term of $e^{x/2}$ equals to that of $e^{x/2}-e^{-x/2}$.
\end{proof}
\vspace{-1mm}

\bibliographystyle{abbrv}

\end{document}